\documentclass[11pt]{amsart}
\usepackage{palatino}
\usepackage[v2]{xy}
\usepackage{amssymb}

\let\oldmarginpar\marginpar
\renewcommand\marginpar[1]{\-\oldmarginpar[\raggedleft\footnotesize #1]%
{\raggedright\footnotesize #1}}

\newtheorem{theorem}{Theorem}[section]
\newtheorem{corollary}[theorem]{Corollary}
\newtheorem{proposition}[theorem]{Proposition}
\theoremstyle{definition}
\newtheorem{definition}[theorem]{Definition}
\newtheorem{remark}[theorem]{Remark}

\usepackage{mathrsfs}

\newcommand{\Cc}{{\mathscr C}}
\newcommand{\Dd}{{\mathscr D}}
\newcommand{\Ee}{{\mathscr E}}
\newcommand{\Bb}{{\mathscr B}}
\newcommand{\Mm}{{\mathscr M}}
\newcommand{\Aa}{{\mathscr A}}
\newcommand{\Gg}{{\mathscr G}}

\usepackage{bbold}

\newcommand{\Z}{{\mathbb Z}}

\newcommand{\op}{{\rm op}}
\newcommand{\Hom}{{\rm Hom}}
\newcommand{\llim}{{\rm lim}}
\newcommand{\clim}{{\rm colim}}

\newcommand{\Ext}{{\rm Ext}}
\newcommand{\Tor}{{\rm Tor}}
\newcommand{\Mor}{{\rm Mor}}

\newcommand{\h}{{\mathcal H}}
\newcommand{\Ner}[1]{{\mathcal N}(#1)}

\def\thomnat{\mathfrak{NatS}^{Th}(\Aa)}
\def\cochain{\mathfrak{CoChn}(\Aa)}
\def\thomcontra{\mathfrak{NatS}_{Th}(\Aa)}
\def\chain{\mathfrak{Chn}(\Aa)}

\newcommand{\Set}{{\mathfrak{Sets}}}
\newcommand{\Cat}{{\mathfrak{Cat}}}
\newcommand{\Fun}{{\mathfrak{Fun}}}
\newcommand{\PsdFun}{{\mathfrak{PsdFun}}}
\newcommand{\Fib}{{\mathfrak{Fib}}}
\newcommand{\Ab}{{\mathfrak{Ab}}}
\newcommand{\Mod}{{\mathfrak{Mod}}}

\begin{document}
\def\ead{\email}

\title{Thomason cohomology of categories
}

\author{Imma G\'alvez-Carrillo}
\address{Departament de Matem\`atica Aplicada III,
      Universitat Polit\`ecnica de Cata\-lunya\\ Escola d'Enginyeria de Terrassa, Carrer Colom 1, 08222 Terrassa (Bar\-celona), Spain}
\ead{m.immaculada.galvez@upc.edu}

\author{Frank Neumann}
\address{Department of Mathematics, University of Leicester\\
University Road, Leicester LE1 7RH, England, United Kingdom}
\ead{fn8@mcs.le.ac.uk}

\author{Andrew Tonks
}
\address{STORM, London Metropolitan University\\
166--220 Holloway Road, London N7 8DB, United Kingdom}
\ead{a.tonks@londonmet.ac.uk}

\begin{abstract}
We introduce cohomology and homology theories for small categories with general coefficient systems from simplex categories first studied by Thomason. These theories generalize at once Baues-Wirsching cohomology and homology and other more classical theories. We analyze naturality and functoriality properties of these theories and construct associated spectral sequences for functors between small categories. 
\end{abstract}

\keywords{Cohomology and homology theories for small categories, Baues-Wirsching cohomology, Thomason natural systems, Grothendieck fibrations}
 
\maketitle
\vspace*{-0.3cm}
\section*{Introduction}

In this article we introduce a new cohomology theory for small categories with very general coefficient systems. For any small category $\Cc$ we   define a cohomology theory with coefficients being functors from the simplex category $\Delta/\Cc$ to abelian categories.~Functors of this kind were studied by Thomason in his notebooks in order to define very general notions of homotopy limits and to analyze their functoriality properties \cite{Wei2}.~It turns out that these general coefficient systems, which we will call {\it Thomason natural sytems}, provide also a systematic way to study the cohomology of small categories, especially with respect to general naturality and functoriality properties.~This cohomology theory, called {\it Thomason cohomology} here, generalizes Baues-Wirsching cohomology, whose coefficients instead are functors from the factorization category $F\Cc$ to abelian categories.~Baues-Wirsching cohomology was introduced in \cite{BW} to study systematically linear extensions of categories appearing in algebra and homotopy theory. ~Dually, we also introduce Thomason homology for small categories using contravariant Thomason natural systems generalizing Baues-Wirsching homology as introduced by the authors in \cite{GNT}. The naturality and functoriality of these theories using functors from simplex categories as coefficients immediately extend the functoriality properties of Baues-Wirsching theories as discussed before in \cite{BW}, \cite{Mu}, \cite{pr} and \cite{GNT}.

In the first part we will introduce Thomason natural systems as functors from the simplex category of a given small category into an abelian category, define the cohomology and homology theories with respect to these general coefficient systems and analyze their basic homological properties.~In the second part we then compare these cohomology and homology theories with other theories studied before in the literature. In particular we will discuss how Thomason cohomology generalizes Baues-Wirsching cohomology and therefore also other classical theories like Hochschild-Mitchell cohomology and cohomology of categories with local and constant coefficient systems. We then analyze the functoriality of Thomason cohomology and homology with respect to a given functor between small categories and construct Leray type spectral sequences in particular situations, including the cases when the functor is part of an adjoint pair and when it is a Grothendieck fibration. This uses the general constructions of spectral sequences for cohomology and homology of small categories as developed systematically in \cite{GNT}.~As a bonus, naturality and functoriality properties of Baues-Wirsching cohomology and other classical cohomology theories of small categories will follow almost immediately. 

Though we will be working here exclusively within the categorical framework, it is possible by using the identification of the simplex category $\Delta/\Cc$ of $\Cc$ with the category of simplices $\Delta/\Ner \Cc$ of the simplical nerve $\Ner \Cc$ to relate Thomason cohomology and homology with the cohomology and homology of simplicial sets and associated Leray-Serre spectral sequences as discussed in \cite{GZ}, \cite{BK}, \cite{Dr} and \cite{I}. The categorical framework presented here has the advantage of being more direct and not relying on any homotopical framework, therefore providing a purely categorical framework for the study of cohomology of simplicial sets.~A systematic study of the relationship between the categorical framework and the simplicial constructions will appear in a sequel to this article.

\section{Thomason (co)homology of categories}

\subsection{Categories of simplices}

Recall that the {\it standard simplex category} $\Delta$ is the category whose objects are the nonempty finite ordered sets $[m]=\{0<1<\cdots<m\}$ and morphisms are order preserving functions. We may regard each $[m]$ as a category and $\Delta$ as the corresponding full subcategory of the category $\Cat$ of small categories.

\begin{definition} 
The {\it simplex category} $\Delta/\Cc$ of a small category $\Cc$ is the category whose object set is the set of pairs $([n],f)$, where $[n]$ is an object of $\Delta$ and $f:[n]\to \Cc$ is a functor, and whose morphisms $([n],f)\longrightarrow ([m],g)$ are morphisms $\sigma:[n]\to[m]$ of $\Delta$ with $f=g\circ \sigma$.
\end{definition}

Thus objects $([n],f)$ of $\Delta/\Cc$ can be interpreted as elements of the simplicial nerve $\Ner \Cc$ of $\Cc$. We will often omit the $[n]$ from the notation and 
regard objects as diagrams
$$f=(C_0\stackrel{f_1}\longleftarrow C_1\stackrel{f_2}\longleftarrow\cdots \stackrel{f_n}\longleftarrow C_n)$$
The morphisms of $\Delta/\Cc$ are generated by omitting or repeating objects $C_i$ in such diagrams.

Equivalently the simplex category of $\Cc$ is the Grothendieck construction of the contravariant diagram of discrete categories given by the simplicial nerve $\Ner \Cc$ of $\Cc$, 
$$\Delta/\Cc\;\;\cong\;\;\int_{\Delta^\op}\Ner\Cc\,\,\,\,\text{where}\,\,\,\,\Ner \Cc:\Delta^\op\to \Set\to\Cat.$$

More generally, if $X$ is an arbitrary simplicial set then its {\em simplex category} $\Delta/X$ is defined as the category whose objects are pairs $([n],x)$ for $x\in X_n$ and whose morphisms $([n],x)\to ([m],y)$ are morphisms  $\sigma:[n]\to[m]$ of $\Delta$ with $x=\sigma^*y$.  The simplex category $\Delta/\Cc$ of a small category $\Cc$ can therefore be identified with the simplex category $\Delta/\Ner \Cc$ of the nerve $\Ner \Cc$. It was shown by Latch~\cite{L} that the functor $\Delta/-: \Delta^{op}\Set\rightarrow \Cat$ from simplicial sets to small categories is homotopy inverse to the nerve functor ${\mathcal N}: \Cat\rightarrow \Delta^{op}\Set$. Therefore constructions using the simplex category $\Delta/\Cc$ correspond to constructions with the category of simplices $\Delta/\Ner \Cc$. 

\subsection{(Co)homology of categories with Thomason natural systems as coefficients}

We define cohomology and homology theories for small categories with very general coefficient systems introduced by Thomason
to define abstract notions of homotopy limits and colimits \cite{Wei2}.

\begin{definition}\label{def:structure}
Let $\Mm$ be a category. A functor $T: \Delta/\Cc\rightarrow \Mm$ is called a {\it Thomason natural system with values in} $\Mm$. 
If $([n],g\circ\sigma)\to([m],g)$ is a morphism in $\Delta/\Cc$ for some
$\sigma:[n]\to [m]$ in $\Delta$, then
$$\sigma_*:T(g\circ \sigma)\to T(g)$$
denotes the induced morphism in $\Mm$.
\end{definition}

Here we will mainly consider Thomason natural systems with values in an abelian category $\Aa$. Particular cases include the category $\Ab$ of abelian groups and the abelian category $R$-$\Mod$ of (left) $R$-modules for a ring $R$.\\

We define now the Thomason cohomology of a small category as follows:

\begin{definition}\label{bwcochn}
Let $\Cc$ be a small category and let $T: \Delta/\Cc\rightarrow \Aa$ be a Thomason natural system with values in a complete abelian category $\Aa$ with exact products. 
We define the {\it Thomason cochain complex} $C_{Th}^*(\Cc, T)$
 by
$$C_{Th }^n(\Cc, T):= \prod_{([n],f)\in \Delta/\Cc} T(f)$$
for each integer $n\geq 0$, with the differential
$$
d=\sum_{i=0}^{n+1}(-1)^i d_i:
\prod_{[n]\stackrel f\to\Cc}
  T(f) \longrightarrow 
\prod_{[n+1]\stackrel g\to\Cc} \!\! T(g).$$
Here the morphism $d_i$ is induced by the coface maps $\delta^i:[n]\to[n+1]$,
$$d_i\biggl((a_{\!f})_{[n]\stackrel f\to\Cc}\biggr)= \biggl(
{\delta^i}_*(a_{\!g\circ\delta^i})
\biggr)_{[n+1]\stackrel g\to\Cc}$$
on all sequences of elements $a_{\!f}\in T(f)$.

The {\it $n$-th Thomason cohomology} is defined as the cohomology of this complex,
$$H^n_{Th}(\Cc, T):=H^n(C_{Th}^*(\Cc, T), d).$$
\end{definition}
Thomason cohomology can also be expressed as the cohomology of the cosimplicial  replacement of the functor $T:\Delta/\Cc\to\Aa$. Recall from \cite{Wei2} that the {\em cosimplicial replacement} of any functor $T:\Delta/\Cc\to\Mm$ is defined as the cosimplicial object  ${\prod}^*T:\Delta\to\Mm$ with
$${\prod}^nT=\prod_{[n]\stackrel f\to\Cc} T(f).$$
For a morphism $\sigma:[n]\to[m]$ in $\Delta$, the 
cosimplicial structure map ${\prod}^nT\to{\prod}^mT$ is defined by requiring that the component corresponding to $g:[m]\to\Cc$ is the projection onto the component corresponding to $g\circ\sigma:[n]\to\Cc$ 
followed by $\sigma_*:T(g\circ\sigma)\to T(g)$.
When $\Mm$ is an abelian category, cohomology of a cosimplicial object is that of the corresponding cochain complex. 

For any abelian category $\Aa$, let $\thomnat$ be 
the category whose objects are the Thomason natural systems $T:\Delta/\Cc\to\Aa$ with values in $\Aa$, in which a morphism $(\varphi,\tau):T_1\to T_2$ between Thomason natural systems $\Delta/\Cc_i\stackrel{T_i}\longrightarrow\Aa$ consists of a functor $\varphi:\Cc_2\to \Cc_1$ and a natural transformation $\tau:T_1\circ\Delta/\varphi\longrightarrow T_2$. The composition of morphisms is indicated in the following diagram,
$$
\xymatrix@R3em@C5em{\Delta/\Cc_1\drto_{T_1}&\lto_{\Delta/\varphi} \dto|(0.45){\stackrel\tau\Longrightarrow\quad T_2\quad\stackrel\upsilon\Longrightarrow}\Delta/\Cc_2&\lto_{\Delta/\psi} \Delta/\Cc_3\dlto^{T_3}\\&\Aa}
$$   
Then the Thomason cochain complex defines in fact a functor 
$$C_{Th}^*:\thomnat\to\cochain, \,\, C^*_{Th}(T):=C^*_{Th}(\Cc, T)$$ 
from the category of Thomason natural systems with values in the abelian category $\Aa$ to the category of cochain complexes of $\Aa$.  
The functor $C^*_{Th}$ is defined on objects as above, and on morphisms by
\begin{align*}C^*_{Th}(\varphi,\tau):C^*_{Th}(\Cc_1,T_1)&\longrightarrow C^*_{Th}(\Cc_2,T_2)\\
(a_f)_{[n]\stackrel f\to\Cc_1}
&\longmapsto
(\tau_g(a_{\varphi\circ g}))_{[n]\stackrel g\to\Cc_2}
\end{align*}
Thomason cohomology is then a functor from $\thomnat$ to the category of graded objects in $\Aa$.\\

Dually, we can define the Thomason homology of a small category: 
\begin{definition}
Let $\Cc$ be a small category and let $T: (\Delta/\Cc)^\op\rightarrow \Aa$ be a contravariant Thomason natural system with values in a cocomplete abelian category $\Aa$ with exact coproducts. 
We define the {\em Thomason chain complex} $C^{Th}_*(\Cc, T)$ by
$$C^{Th}_n(\Cc, T):= \bigoplus_{([n],f)\in\Delta/\Cc} T(f)$$
for each integer $n\geq 0$, with the differential
\begin{eqnarray*}
\bigoplus_{[n+1]\stackrel f\to\Cc}
  \!\!\!\!
T(f) &\stackrel d \longrightarrow &
\bigoplus_{[n]\stackrel g\to\Cc} \!\! T(g)\\
a_f&\mapsto&\sum_{i=0}^{n+1}(-1)^i(\delta^i)^*
(a_{f})
\end{eqnarray*}
where $(\delta^i)^*:T(f)\to T(f\delta^i)$ is induced by the coface map $\delta^i:[n]\to[n+1]$.
The {\it $n$-th Thomason homology} is defined as
$$H_n^{Th}(\Cc, T):=H_n(C^{Th}_*(\Cc, T), d).$$
\end{definition}

This can also be expressed as the homology of the simplicial replacement ${\coprod}_*T:\Delta^\op\to\Aa$ of the functor $T:(\Delta/\Cc)^\op\to\Aa$, where
$${\coprod}_nT=\bigoplus_{[n]\stackrel f\to\Cc} T(f)$$
and, for a morphism $\sigma:[m]\to[n]$ in $\Delta$, the simplicial structure map ${\coprod}_nT\to{\coprod}_mT$ 
is defined on the component corresponding to $f:[n]\to\Cc$  by $\sigma^*:T(f)\to T(f\circ \sigma)$.

Let $\thomcontra$ be the category with objects the {\it contravariant} Thomason natural systems $T:(\Delta/\Cc)^\op\to\Aa$, in which a morphism $(\varphi,\tau):T_1\to T_2$ is given by a functor $\varphi:\Cc_1\to\Cc_2$ together with a natural transformation $\tau:T_1\to T_2\circ\Delta/\varphi$. 
The composition of morphisms is indicated in the following diagram,
$$
\xymatrix@R3em@C5em{(\Delta/\Cc_1)^{op}\drto_{T_1}\rto^{\Delta/\varphi}& \dto|(0.45){\stackrel\tau\Longrightarrow\quad T_2\quad\stackrel\upsilon\Longrightarrow}(\Delta/\Cc_2)^{op}\rto^{\Delta/\psi}& (\Delta/\Cc_3)^{op}\dlto^{T_3}\\&\Aa}
$$   
The Thomason chain complex defines a functor
$$C^{Th}_*:\thomcontra\to \chain, \,\, C_*^{Th}(T):=C_*^{Th}(\Cc, T).$$
where for morphisms we define
$$C^{Th}_*(\varphi,\tau):C^{Th}_*(\Cc_1, T_1)\longrightarrow C^{Th}_*(\Cc_2, T_2)$$ 
using the maps  
$$
\tau_f:T_1(f)\longrightarrow T_2(\varphi\circ f).
$$

\subsection{Basic properties of Thomason (co)homology}

\def\symbhom{\overline\Hom}
\def\symbext{\Ext}
\def\symbtor{\Tor}

Let $T:\Delta/\Cc\to \Aa$ be a Thomason natural system with values in a complete abelian category $\Aa$ with exact products. There is a contravariant functor from $\Aa$ to the functor category of $\Ab$-valued Thomason natural systems, 
$$
\Hom_\Aa(-,T(-)):
\Aa\longrightarrow 
\Fun(\Delta/\Cc,\Ab)
^{op}, 
$$
defined on objects $A$ of $\Aa$ by the functors $\Hom_\Aa(A,T(-)):\Delta/\Cc\to\Ab$ and on morphisms $f: A\rightarrow B$ in $\Aa$ by the induced natural transformations $\eta_f: \Hom_\Aa(B,T(-))\rightarrow \Hom_\Aa(A,T(-))$.

The right adjoint of this functor is
$$\symbhom_{\Delta/\Cc}(-,T):
\Fun(\Delta/\Cc,\Ab)
^{op} \longrightarrow \Aa,
$$
and varying $T$ gives the bifunctor
$$\symbhom_{\Delta/\Cc}(-,-):
\Fun(\Delta/\Cc,\Ab)
^{op}
\times 
\Fun(\Delta/\Cc,\Aa)
\longrightarrow \Aa,
$$
termed the \emph{symbolic hom} functor (see \cite{Fr}, \cite[Remark 2.3]{Hu}). 
\begin{proposition}\label{symb} The symbolic hom functor satisfies:
\begin{enumerate}
\item
If $\Z\Hom_{\Delta/\Cc}(-,-):(\Delta/\Cc)^{op}\times \Delta/\Cc\rightarrow \Ab$ is the bifunctor sending  
$(f, g)$ to the free abelian group on the set $\Hom_{\Delta/\Cc}(f, g)$,
then there is a natural isomorphism
$$\symbhom_{\Delta/\Cc}({{\Z\Hom}}_{\Delta/\Cc}(f, -), T(-))\cong
T(f).
$$
\item
If $\underline \Z:\Delta/\Cc\to \Ab$ is the constant Thomason natural system,~i.~e. the functor with constant value the abelian group $\Z$, then there is a natural isomorphism
$$
\symbhom_{\Delta/\Cc}(\underline \Z,T)\cong {\lim}_{\Delta/\Cc} T.
$$
\end{enumerate}
\end{proposition}
\begin{proof}
These results are well known in the special case $\Aa=\Ab$, and the general case follows by standard arguments using the adjunction. For (1), morphisms from an object $A$ of $\Aa$ to the left hand side are in bijection with natural transformations between the two functors $\Z\Hom_{\Delta/\Cc}(f,-),\:\Hom_\Aa(A,T(-))):\Delta/\Cc\to\Ab$, which are in turn in bijection with the morphisms from $A$ to the right hand side (by Yoneda). The natural isomorphism in (1) follows. The proof of (2) is similar, and may be found for example in \cite[Lemma 4.3.1]{liege} or \cite{Ob}.
\end{proof}
The corresponding derived functors are therefore also naturally isomorphic (see \cite[III.1, III.7]{Ob}, \cite[Sect. 2]{Hu}),~i.~e.
$$
\symbext^n_{\Delta/\Cc}(\underline{\Z}, T) \cong{\lim}_{\Delta/\Cc}^n T.
$$

In the case $\Aa=\Ab$ this gives the usual Ext-functor for diagrams of abelian groups over $\Delta/\Cc$.

\begin{remark}\label{remten}
Though above we have worked for convenience only with Thomason natural systems, we can define for any small category $\Cc$ and a complete abelian category with exact products $\Aa$ the {\it symbolic hom functor}
$$\symbhom_{\Cc}(-,-):
\Fun(\Cc,\Ab)
^{op}
\times 
\Fun(\Cc,\Aa)
\longrightarrow \Aa.
$$
It will have the same properties as in Theorem \ref{symb}, especially we get an isomorphism
$$\symbhom_{\Cc}(\underline \Z, D)\cong {\lim}_{\Cc} D$$  
for the constant functor ${\underline \Z}: \Cc\rightarrow \Ab$ and for any functor $D: \Cc\rightarrow \Aa$,
which extends to an isomorphism of the associated derived functors, i. e.
$$
\symbext^n_{\Cc}(\underline{\Z}, D) \cong{\lim}_{\Cc}^n D.
$$

Dually, if $\Aa$ is a cocomplete abelian category with exact coproducts, we have the {\it symbolic tensor product functor}
$$-\overline{\otimes}_{\Cc}-: 
\Fun(\Cc^{op},\Ab)
\times 
\Fun(\Cc,\Aa)
\longrightarrow \Aa
.$$
We have an isomorphism
$$\underline{\Z}\overline{\otimes}_{\Cc}D\cong {\clim}^{\Cc}D,$$
which extends to an isomorphism of the associated derived functors
$$\symbtor_n^{\Cc}(\underline{\Z}, D)\cong {\clim}_n^{\Cc}D.$$


We refer the reader to \cite[III.1, III.7]{Ob} for the general theory and proofs (see also \cite[Sect. 2]{Hu}, \cite{ML1}).
\end{remark}

We can now characterize Thomason cohomology and homology of small categories as cohomology and homology of the simplex category $\Delta/\Cc$.

\begin{theorem}\label{THExt}
Let $\Aa$ be complete abelian category with exact products.
There are isomorphisms, natural in $\Cc$ and in $T: \Delta/\Cc\rightarrow \Aa$, such that
$$H^n_{Th}(\Cc, T)\cong
\symbext^n_{\Delta/\Cc}(\underline{\Z}, T) \cong {\lim}^n_{\Delta/\Cc}T
=H^n(\Delta/\Cc, T).$$
\end{theorem}

\begin{proof}
We construct a chain complex $B_*=\{B_n, d_n\}$, the {\it generalized bar resolution} of the constant functor $\underline{\Z}$ in the functor category $\Fun(\Delta/\Cc, \Ab)$.

Let $B_*=\coprod_*{{\Z\Hom}}_{\Delta/\Cc}$ be the simplicial replacement of the functor $${{\Z\Hom}}_{\Delta/\Cc}: (\Delta/\Cc)^{op}\rightarrow \Fun(\Delta/\Cc, \Ab), \quad f\mapsto {{\Z \Hom}}_{\Delta/\Cc}(f, -).$$ 
so that, by definition of the simplicial replacement, we have
$$B_n={\coprod}_n {{\Z\Hom}}_{\Delta/\Cc}=\bigoplus_{[n]\stackrel{f}\rightarrow \Cc}{{\Z\Hom}}_{\Delta/\Cc}(f, -): \Delta/\Cc\rightarrow \Ab.$$

\noindent We have to show the following:
\begin{enumerate}
\item If $T: \Delta/\Cc\rightarrow \Aa$ is a Thomason natural system and  ${\prod}^*T:\Delta\to\Aa$ is its cosimplicial replacement, there is a natural isomorphism $$\symbhom_{\Delta/\Cc}(B_*, T)\;\cong \;{\prod}^*T.$$ 
\item The complex $B_*$ is a free resolution of the constant Thomason natural system $\underline{\Z}:\Delta/\Cc\to\Ab$.
\end{enumerate}
It will then follow that
$$
\symbext^*_{\Delta/\Cc}(\underline{\Z},T)\cong
H^*(\symbhom_{\Delta/\Cc}(B_*,T))
\cong H^*({\textstyle\prod}^*T)= H^*_{Th}(\Cc,T)
$$
as required.

The first part $(1)$ holds since
$$
\symbhom_{\Delta/\Cc}(B_n, T)\;\cong\! 
\prod_{[n]\stackrel{f}\rightarrow \Cc} \symbhom_{\Delta/\Cc}({{\Z\Hom}}_{\Delta/\Cc}(f, -), T(-))\;\cong\!
\prod_{[n]\stackrel{f}\rightarrow \Cc}T(f).
$$
For the second part (2) we observe that
each $B_n(g)$, for $g:[m]\to\Cc$ is the free abelian group on the set 
$\bigcup_{[n]\stackrel{f}\rightarrow \Cc}\Hom_{\Delta/\Cc}(f, g)\cong \Delta([n],[m])$.
That is, each $B_*(g)$ is the contractible simplicial abelian group $\Z\Delta[m]$. Thus $B_*$ is projective and weakly equivalent to $\underline\Z$ in the functor category
$\Fun(\Delta/\Cc, \Ab)$.
\end{proof}

Dually, we have the following identification of Thomason homology of small categories:

\begin{theorem}\label{THTor}
Let $\Aa$ be cocomplete abelian category with exact coproducts.
There are isomorphisms, natural in $\Cc$ and in $T: (\Delta/\Cc)^{op}\rightarrow \Aa$, 
such that
$$H_n^{Th}(\Cc, T)\cong \symbtor_n^{(\Delta/\Cc)^{op}}(\underline{\Z}
, T) \cong{\clim}_n^{(\Delta/\Cc)^{op}}T=H_n((\Delta/\Cc)^{op}, T).$$
\end{theorem}

\begin{proof} This follows along the same lines as Theorem \ref{THExt} using the generalized bar resolution of the constant functor 
$\underline{\Z}$ involving the dual notions, namely the symbolic tensor product functor and its derived Tor-functor as defined in Remark \ref{remten} for functors $T: (\Delta/\Cc)^{op}\rightarrow \Aa$.
\end{proof}

The following proposition describes the basic functoriality properties of Thomason cohomology with respect to a pair of adjoint functors between small categories:

\begin{proposition}\label{adjequiv}Let $\Cc$ and $\Dd$ be small categories, $\Aa$ a complete abelian category with exact products and $(\varphi, \psi)$ a pair of adjoint functors:
$$\varphi: \Dd \rightleftarrows \Cc :\psi.$$
Then for any Thomason natural system $T$ on $\Cc$ there is a natural isomorphism
$$H^*_{Th}(\Dd, \varphi^*(T))\cong H^*_{Th}(\Cc, T).$$
In particular, an equivalence of categories $\varphi:\Dd\to\Cc$ induces a natural isomorphism 
 $$H^*_{Th}(\Dd, \varphi^*(T))\cong H^*_{Th}(\Cc, T)$$
for any Thomason natural system $T$ on $\Cc$.
\end{proposition}

\begin{proof} We observe that the adjoint pair $(\varphi,  \psi)$ induces an adjoint pair $(\Delta/\varphi, \Delta/\psi)$ of functors between the associated simplex categories $\Delta/\Cc$ and $\Delta/\Dd$:
$$\Delta/\varphi: \Delta/\Cc\rightleftarrows \Delta/\Dd :\Delta/\psi.$$
The statement now follows from Theorem 1.5 above and \cite[Lemma 1.5, p. 10] {FFPS} with ${\mathbf C}=\Delta/\Cc$, ${\mathbf D}=\Delta/\Dd$ and $F=\underline{\Z}$ the constant Thomason natural system.

In particular, by \cite[Theorem IV.4.1]{ML2} an equivalence of categories $\varphi$ is part of an adjunction between functors and so the second statement follows. \end{proof}

Dually, using similar arguments, we have the following results regarding the basic functoriality of Thomason homology:

\begin{proposition}\label{adjequivhom}Let $\Cc$ and $\Dd$ be small categories, $\Aa$ a cocomplete abelian category with exact coproducts and $(\varphi, \psi)$  a pair of adjoint functors:
$$\varphi: \Dd \rightleftarrows \Cc :\psi.$$
Then for any contravariant Thomason natural system $T$ on $\Cc$ there is a natural isomorphism
$$H_*^{Th}(\Dd, \varphi^*(T))\cong H_*^{Th}(\Cc, T).$$
In particular, an equivalence of categories $\varphi:\Dd\to\Cc$ induces a natural isomorphism 
 $$H_*^{Th}(\Dd, \varphi^*(T))\cong H_*^{Th}(\Cc, T)$$
for any contravariant Thomason natural system $T$ on $\Cc$.
\end{proposition}


\section{Thomason (co)homology and spectral sequences}

\subsection{Thomason (co)homology in relation to other theories}
Thomason cohomology and homology defined in the first part generalize many other constructions of cohomology and homology for small categories in the literature. The coefficient systems of these theories can all be interpreted as particular cases of Thomason natural systems. 

Let us discuss the most important examples here (see also \cite{BW, GNT}).  Recall first that the
{\it factorization category} $F\Cc$ of a small category $\Cc$ is the category whose object set is the set of morphisms of $\Cc$ and whose Hom-sets $F\Cc(f, f')$ are the sets of pairs $(\alpha, \beta)$ such that $f'=\beta f\alpha$,
 $$
\xymatrix{b\rto^\beta &b'\\
a\uto^{f}& a'\uto_{f'}\lto^\alpha.}
$$

There is a functor $$\nu:\Delta/\Cc\to F\Cc$$  defined on objects and morphisms by
\begin{align*}
\nu
\left(
C_0\stackrel{f_1}\longleftarrow C_1\stackrel{f_2}\longleftarrow\cdots \stackrel{f_m}\longleftarrow C_m
\right)
&
= 
(
C_0\xleftarrow{f_1\circ\cdots\circ f_m} C_m)
\\
\nu
\left(([m],f)
\stackrel\sigma\longrightarrow ([n],g)\right)
&
= 
(
C_m\xleftarrow{g_{\sigma(m)+1}\circ\dots\circ g_n}D_n
,\,
D_0\xleftarrow{
g_{1}\circ\dots\circ g_{\sigma(0)}}C_0
)
\end{align*}
In the cases $\sigma(m)=n$ or $\sigma(0)=0$ the terms on the right hand side of the last equation are the identity maps $C_m\leftarrow D_n$ and $D_0\leftarrow C_0$ respectively. 

We now have the following diagram of six categories and functors between them, in which $\Cc$ is a fixed small category:
$$\Delta/\Cc\stackrel {\nu}\longrightarrow F\Cc \stackrel{\pi}\longrightarrow\Cc^{op}\times \Cc\stackrel{p}\longrightarrow \Cc\stackrel{q}\longrightarrow\pi_1\Cc\stackrel{t}\longrightarrow \mathbb{1}.$$
In this diagram, $\nu$ is the functor introduced above, $\pi$ is the forgetful functor, $p$ is the projection to the second factor and $q$ is the localization functor into the fundamental groupoid $\pi_1\Cc=(\Mor \Cc)^{-1}\Cc$ of $\Cc$ (see \cite{GZ}). Furthermore, $\mathbb{1}$ is the terminal category, consisting of one object and one morphism, and $t$ is the canonical functor.

Now let $\Aa$ be a (co)complete abelian category with exact (co)products, as above. 
Pre-composition with the functor $\nu$ induces a functor between functor categories
$$\nu^*: \Fun(F\Cc, \Aa)\rightarrow \Fun(\Delta/\Cc, \Aa),$$
and, more generally, pulling back functors from $\Fun(\Cc', \Aa)$ via the functors in the above diagram by pre-composition, where $\Cc'$ denotes any of the six categories in it, induces Thomason natural systems on the category $\Cc$. We can therefore define various versions of cohomology and homology of small categories for different types of coefficient systems, which can then all be seen as special cases of Thomason cohomology or homology. In some sense Thomason cohomology and homology are the most general cohomology and homology theories one can define for small categories as they generalize especially Baues-Wirsching cohomology and homology. It is known that Baues-Wirsching cohomology generalizes the classical cohomology theories studied for example by Watts \cite{Wa}, Mitchell \cite{Mi}, Quillen \cite{Q} and others (see \cite[Definition 1.18]{BW} and \cite[Definition 1.11]{GNT}). Particular examples include group and groupoid cohomology and homology.

\begin{definition}\label{DefBW}
Let $\Cc$ be a small category and $\Aa$ a complete abelian category with exact products for cohomology, or a cocomplete abelian category with exact coproducts for homology. Then:
\begin{itemize}
\item[(i)] $D$ is called a {\it Baues-Wirsching natural system} if $D$ is a functor of $\Fun(F\Cc, \Aa)$. Define the cohomology $H_{BW}^*(\Cc, D)=H^*_{Th}(\Cc, \nu^*D)$ and dually the homology $H^{BW}_*(\Cc, D)=H_*^{Th}(\Cc, \nu^{op*}D)$.
\item[(ii)] $M$ is called a {\it $\Cc$-bimodule} if $M$ is a functor of $\Fun(\Cc^{op}\times \Cc, \Aa)$. Define the cohomology $H_{HM}^*(\Cc, M)=H^*_{Th}(\Cc, \nu^*\pi^*M)$ and dually the homology $H^{HM}_*(\Cc, M)=H_*^{Th}(\Cc, \nu^{op*}\pi^{op*}M)$.
\item[(iii)] $F$ is called a {\it $\Cc$-module} if $F$ is a functor of $\Fun(\Cc, \Aa)$. Define the cohomology $H^*(\Cc, F)=H^*_{Th}(\Cc, \nu^*\pi^*p^*F)$ and dually the homology $H_*(\Cc, F)=H_*^{Th}(\Cc, \nu^{op*}\pi^{op*}p^{op*}F)$.
\item[(iv)] $L$ is called a {\it local system} on $\Cc$ if $L$ is a functor of $\Fun(\pi_1\Cc, \Aa)$. Define the cohomology $H^*(\Cc, L)=H^*_{Th}(\Cc, \nu^*\pi^*p^*q^*L)$ and dually the homology $H_*(\Cc, L)=H_*^{Th}(\Cc, \nu^{op*}\pi^{op*}p^{op*}q^{op*}L)$.
\item[(v)] $A$ is a {\it trivial system} on $\Cc$ if $A$ is an abelian group (resp. an object in $\Aa$), i.~e. a functor of $\Fun(\mathbb{1}, \Ab)$ (resp. a functor of  $\Fun(\mathbb{1}, \Aa)$). Define the cohomology $H^*(\Cc, A)=H^*_{Th}(\Cc, \nu^*\pi^*p^*q^*t^*A)$ and dually the homology $H_*(\Cc, A)=H_*^{Th}(\Cc, \nu^{op*}\pi^{op*}p^{op*}q^{op*}t^{op*}A)$.
\end{itemize}
\end{definition}

Note that the functors $\nu$, $\pi$, $p$, $q$, $t$ above form natural transformations between the different
endofunctors of $\Cat$, since for any 
functor $\varphi:\Dd\to\Cc$
we have the following commutative ladder:
$$
\xymatrix@C=14pt{\Delta/\Dd\rrto^-{\nu}\dto_{\Delta/\varphi}&&F\Dd\rrto^-{\pi}\dto^{F\varphi}&&\Dd^{op}\times \Dd\dto^{\varphi^{op}\times \varphi}\rrto^-{p}&& \Dd\dto^\varphi\rrto^{q}&&\pi_1\Dd\dto^{\pi_1\varphi}\rrto^{t}&&\mathbb{1}\ar@{=}[d]\\
   \Delta/\Cc\rrto^-{
{\nu}}&&              F\Cc\rrto^-{
{\pi}}               &&\Cc^{op}\times \Cc\rrto^-{
{p}}&&\Cc\rrto^-{
{q}}            &&\pi_1\Cc\rrto^-{
{t}}      &&\mathbb{1}.\!\!\!}
$$

The various cohomology and homology theories can now be identified with the ones known previously in the literature. We have the following result:

\begin{theorem} \label{ThmBW} 
Let $\Cc$ be a small category and $\Aa$ a complete abelian category with exact products. Then:
\begin{itemize}
\item[(i)] (Baues-Wirsching) For any natural system $D$ of $\Fun(F\Cc, \Aa)$ we have isomorphisms, natural in $\Cc$ and $D$:\\
$H_{BW}^n(\Cc, D)\cong \Ext^n_{\Delta/\Cc}(\underline{\Z}, \nu^*D)\cong \Ext^n_{F\Cc}(\underline{\Z}, D).$
\item[(ii)] (Hochschild-Mitchell) For any $\Cc$-bimodule $M$ of $\Fun(\Cc^{op}\times \Cc, \Aa)$ we have isomorphisms, natural in $\Cc$ and $M$:\\
$H_{HM}^n(\Cc, M)\cong \Ext_{\Delta/\Cc}^n(\underline{\Z}, \nu^*\pi^*M)\cong\Ext^n_{\Cc^{op}\times \Cc}(\underline{\Z}\Cc, M).$
\item[(iii)] For any $\Cc$-module $F$ of $\Fun(\Cc, \Aa)$ we have isomorphisms, natural in $\Cc$ and $F$:\\
$H^n(\Cc, F)\cong \Ext_{\Delta/\Cc}^n(\underline{\Z}, \nu^*\pi^*p^*F)\cong\Ext^n_{\Cc}(\underline{\Z}, F)\cong{\lim}^n_{\Cc} F.$
\item[(iv)] For any local system $L$ on $\Cc$ of $\Fun(\pi_1\Cc, \Aa)$ we have isomorphisms, natural in $\Cc$ and $L$:\\
$H^n(\Cc, L)\cong \Ext_{\Delta/\Cc}^n(\underline{\Z}, \nu^*\pi^*p^*q^*L)\cong \Ext^n_{\pi_1\Cc}(\underline{\Z}, L)\\\hspace*{1.7cm}\cong {\lim}^n_{\pi_1\Cc}L\cong H^n(B\Cc, L)$\\
where $B\Cc$ is the classifying space of the small category $\Cc$.
\end{itemize}
\end{theorem} 

\begin{proof}
The statement of (i) follows directly from the description of the associated cosimplicial replacements for Thomason and Baues-Wirsching natural systems $T$ and $D$, which give rise to the respective Thomason and Baues-Wirsching cochain complexes (see Definition \ref{bwcochn} and for the Baues-Wirsching cochain complex \cite[Definition 1.4]{BW}, \cite[Definition 1.8]{GNT}). Both types of natural systems give rise to the same cosimplicial object (\cite[Proposition 2.8]{Wei2}) and therefore induce isomorphic cohomologies. The stated isomorphisms now follows from \ref{THExt} and Remark \ref{remten} by using the symbolic Hom and its derived Ext-functor in the case of functors from the factorization category $F\Cc$. For $\Aa=\Ab$ the isomorphism between Baues-Wirsching cohomology and the Ext-group was already proved in \cite[Proposition 8.5]{BW}.  

Naturality of the isomorphisms follows now from the naturality diagram above, that is, 
$\nu^*(F\varphi)^*D=(\Delta/\varphi)^*{\nu}^*D$, for any functor $\varphi: \Dd\rightarrow \Cc$ and any Baues-Wirsching natural system $D$ of $\Fun(F\Cc, \Aa)$.

The statements (ii), (iii) and (iv) also follow directly from the characterization of Thomason cohomology as in Theorem \ref{THExt} and as in part (i) using characterizations of the respective cohomology theories via the symbolic Hom and its derived Ext-functors utilizing Remark \ref{remten} for functors from $\Cc^{op}\times \Cc$, $\Cc$, and $\pi_1\Cc$. 

Naturality of the isomorphisms follows similarly to (i), from the naturality ladder above using the appropriate square depending on which case (ii), (iii) or (iv) we are considering.
\end{proof}

Dually, we have a similar characterization for the various homology theories of a small category:

\begin{theorem} Let $\Cc$ be a small category and $\Aa$ a cocomplete abelian category with exact coproducts. Then:
\begin{itemize}
\item[(i)] (Baues-Wirsching) For any natural system $D$ of $\Fun((F\Cc)^{op}, \Aa)$ we have isomorphisms, natural in $\Cc$ and $D$:\\
$H^{BW}_n(\Cc, D)\cong \Tor_n^{(\Delta/\Cc)^{op}}(\underline{\Z}, {\nu^{op }}^* D)\cong \Tor_n^{(F\Cc)^{op}}(\underline{\Z}, D).$
\item[(ii)] (Hochschild-Mitchell) For any $\Cc$-bimodule $M$ of $\Fun(\Cc\times \Cc^{op}, \Aa)$ we have isomorphisms, natural in $\Cc$ and $M$:\\
$H^{HM}_n(\Cc, M)\cong \Tor^{(\Delta/\Cc)^{op}}_n(\underline{\Z}, {\nu^{op}}^*{\pi^{op}}^*M)\cong\Tor_n^{\Cc \times \Cc^{op}}(\underline{\Z}\Cc, M).$
\item[(iii)] For any $\Cc$-module $F$ of $\Fun(\Cc^{op}, \Aa)$ we have isomorphisms, natural in $\Cc$ and $F$:\\
$H_n(\Cc, F)\cong \Tor^{(\Delta/\Cc)^{op}}_n(\underline{\Z}, {\nu^{op}}^*{\pi^{op}}^*{p^{op}}^*F)\cong\Tor_n^{\Cc^{op}}(\underline{\Z}, F)\\\hspace*{1.7cm}\cong{\clim}_n^{\Cc^{op}} F.$
\item[(iv)] For any local system $L$ on $\Cc$ of $\Fun(\pi_1\Cc, \Aa)$ we have isomorphisms, natural in $\Cc$ and $L$:\\
$H_n(\Cc, L)\cong \Tor^{(\Delta/\Cc)^{op}}_n(\underline{\Z}, {\nu^{op}}^*{\pi^{op}}^*{p^{op}}^*{q^{op}}^*L)\cong \Tor_n^{\pi_1\Cc}(\underline{\Z}, L)\\\hspace*{1.7cm}\cong {\clim}_n^{\pi_1\Cc}L\cong H_n(B\Cc, L)$\\
where $B\Cc$ is the classifying space of the small category $\Cc$.
\end{itemize}
\end{theorem}

\begin{proof}
The proofs are just dual to those of the preceding theorem and follow from the characterization of Thomason homology in Theorem \ref{THTor} and using the symbolic tensor product functor and its derived Tor-functor for functors from the various categories $(F\Cc)^{op}$, $\Cc\times \Cc^{op}$, $\Cc^{op}$ and $\pi_1\Cc$.
\end{proof}

As an immediate consequence we obtain the following theorem on the basic functoriality of Baues-Wirsching cohomology due to Muro \cite{Mu} and proved also by Pirashvili and Redondo \cite{pr}:

\begin{corollary} Let $\Cc$ and $\Dd$ be small categories, $\Aa$ a complete abelian category with exact products and $(\varphi, \psi)$ a pair of adjoint functors:
$$\varphi: \Dd \rightleftarrows \Cc :\psi.$$
Then for Baues-Wirsching natural systems $C$ on $\Cc$ and $D$ on $\Dd$ there are natural isomorphisms 
$$H^*_{BW}(\Dd, \varphi^*(C))\cong H^*_{BW}(\Cc, C)$$
$$H^*_{BW}(\Cc, \psi^*(D))\cong H^*_{BW}(\Dd, D).$$
\end{corollary}

\begin{proof}
This follows from the definition of Baues-Wirsching cohomology as induced by Thomason cohomology in Definition \ref{DefBW} and the general functoriality properties of Thomason cohomology in Theorem \ref{adjequiv}.
\end{proof}

Dually, we have also a version of this functoriality property for Baues-Wirsching homology using Theorem \ref{adjequivhom}.
\begin{corollary}
Let $\Cc$ and $\Dd$ be small categories, $\Aa$ a cocomplete abelian category with exact coproducts and $(\varphi, \psi)$ a pair of adjoint functors:
$$\varphi: \Dd \rightleftarrows \Cc :\psi.$$
Then for contravariant Baues-Wirsching natural systems $C$ on $\Cc$ and $D$ on $\Dd$ there are natural isomorphisms 
$$H_*^{BW}(\Dd, \varphi^*(C))\cong H_*^{BW}(\Cc, C)$$
$$H_*^{BW}(\Cc, \psi^*(D))\cong H_*^{BW}(\Dd, D).$$
\end{corollary}

Baues and Wirsching introduced {\it Baues-Wirsching cohomology} in \cite{BW} to study linear extensions of categories.~A homology version was introduced by the authors in \cite{GNT}. 
The cohomology and homology theories for $\Cc$-bimodule coefficients in Theorem \ref{ThmBW} (ii) are the ones normally referred to as {\it Hochschild-Mitchell cohomology} \cite{CWMi, Mi}.  Special cases of Theorem \ref{ThmBW} (iii) have been studied by Watts \cite{Wa}, Roos \cite{Ro} and Quillen \cite{Q}. 

All these cohomology and homology theories enjoy the basic functoriality properties as in the corollaries above. They also generalize cohomology and homology of groupoids $\Gg$, where a groupoid is a category in which all morphisms are isomorphisms. Therefore they in turn also generalize group cohomology and homology, where a group $G$ is interpreted as a category with one object and morphism set $G$. 

\subsection{Functoriality of Thomason (co)homology and spectral sequences}
Given a functor between small categories, we will now derive several spectral sequences for Thomason cohomology and homology. 

Let us consider a functor $u:\Ee\to \Bb$ between small categories and the following commutative ladder of categories and functors:
$$
\xymatrix@C=14pt{\Delta/\Ee\rrto^-{\nu}\dto_{\Delta/u}&&F\Ee\rrto^-{\pi}\dto^{Fu}&&\Ee^{op}\times \Ee\dto^{u^{op}\times u}\rrto^-{p}&& \Ee\dto^u\rrto^{q}&&\pi_1\Ee\dto^{\pi_1u}\rrto^{t}&&\mathbb{1}\ar@{=}[d]\\
   \Delta/\Bb\rrto^-{
{\nu}}&&              F\Bb\rrto^-{
{\pi}}               &&\Bb^{op}\times \Bb\rrto^-{
{p}}&&\Bb\rrto^-{
{q}}            &&\pi_1\Bb\rrto^-{
{t}}      &&\mathbb{1}}
$$


We deal with all the different types of coefficient systems at once in writing $\Cc$ for any of the categories in the lower row of the above ladder and denote by $u':\Delta/\Ee\to \Cc$ the associated composition of functors. Furthermore, let $\Aa$ be a complete abelian category with exact products. 
We get in each case a diagram of the form:
 $$
\xymatrix{
*+++{\Fun(\Delta/\Ee,\Aa)}\ar@<-5pt>[rrrr]_{u'_*} \ar@<5pt>[rrdd]^-{\lim_{\Delta/\Ee}}
&&&& *+++{\Fun(\Cc,\Aa)} \ar@<-5pt>[llll]_{{u'}^*}   \ar@<5pt>[lldd]^-{\lim_{\Cc}} \\ \\
&&\Aa\ar@<5pt>[rruu]^-c \ar@<5pt>[lluu]^-c
}
$$
where $c$ denotes the constant diagram functor, ${u'}^*$ is pre-composition with $u'$, and the other functors in the diagram are the right adjoints of these, given by the limits $\lim_{\Delta/\Ee}$, $\lim_{\Cc}$ and by $u'_*=Ran_{u'}$. 

We obtain a spectral sequence for the derived functors of the composite functor 
$$\llim_{\Delta/\Ee}(-)=\llim_\Cc u'_ *(-)$$
which is an Andr\'e spectral sequence as constructed by the authors in generality in \cite[Section 1.1]{GNT} (see also \cite{An}). Here it converges to the Thomason cohomology of $\Ee$ with coefficients $T$ from $\Fun(\Delta/\Ee,\Aa)$ being a Thomason natural system. Therefore \cite[Theorem 1.2]{GNT} gives a first quadrant cohomology spectral sequence of the form:
$$
E_2^{p,q}\cong H^p(\Cc, Ran^q_{u'}(T))\Rightarrow H^{p+q}_{Th}(\Ee, T)
$$
Here $Ran^q_{u'}$ is the $q$-th right satellite of $Ran_{u'}$, the right Kan extension along the functor $u'$.
Finally, identifying the terms in the $E_2$-page of the spectral sequence using \cite[Corollary 1.3]{GNT} gives us:
$$
E_2^{p,q}\cong H^p(\Cc,  \h^q(-/u', T\circ Q^{(-)}))
\Rightarrow H^{p+q}_{Th}(\Ee, T)
$$
For each object $c$ of $\Cc$ recall that 
$\h^q(c/u', T\circ Q^{(c)})$ is the derived limit
$${\lim}^q\left( c/u'\stackrel{Q^{(c)}}\longrightarrow 
\Delta/\Ee\stackrel{T}\longrightarrow \Aa\right)$$
and $Q^{(c)}:c/u'\to \Delta/\Ee$ is 
the forgetful functor.

In particular cases the $E_2$-page can be simplified.
For example in the case $\Cc=\Delta/\Bb$ with $u'=\Delta/u$ considered above, we get the
following Leray type spectral sequence:

\begin{theorem}\label{H^*T} 
Let $\Ee$ and $\Bb$ be small categories and $u: \Ee\rightarrow \Bb$ a functor. Let $\Aa$ be a complete abelian category with exact products. Given a Thomason natural system $T: \Delta/\Ee\rightarrow \Aa$ on $\Ee$, there is a first quadrant cohomology spectral sequence 
$$E_2^{p,q}\cong H_{Th}^p(\Bb, (R^q \Delta/u_*)(T))\Rightarrow H^{p+q}_{Th}(\Ee, T)$$
which is functorial with respect to natural transformations and where $R^q \Delta/u_*=Ran^q_{\Delta/u}$ is the $q$-th right satellite of $Ran_{\Delta/u}$, the right Kan extension along the induced functor  $\Delta/u$ between the simplex categories.
\end{theorem}

With the identification of the terms in the $E_2$-page above we get therefore:

\begin{corollary}\label{H^*T+GZ} 
Let $u: \Ee\rightarrow \Bb$ be a functor between small categories and $\Aa$ a complete abelian category with exact products. 
Let $T:\Delta/\Ee\rightarrow \Aa$ be a Thomason natural system on $\Ee$.
Then there exists a first quadrant cohomology spectral sequence of the form
$$E_2^{p,q}\cong H_{Th}^p(\Bb, \h^q(-/\Delta/u, T\circ Q^{(-)}))\Rightarrow H_{Th}^{p+q}(\Ee, T)$$
which is functorial with respect to natural transformations and where  
$$\h^q(-/\Delta/u, T\circ Q^{(-)})= \llim^q_{-/\Delta/u}(T\circ Q^{(-)}): \Delta/\Bb\rightarrow \Aa.$$
\end{corollary}

In the special situation that the functor $u$ is part of an adjoint pair of functors or an equivalence of categories this spectral sequence collapses in the $E_2$-page and we simply recover Proposition \ref{adjequiv}. 

For the particular case $\Cc=F\Bb$ and 
$$u'=
{\nu}\circ \Delta/u=Fu\circ \nu:
\Delta/\Ee\longrightarrow\Cc
$$ with a given Thomason natural system $T: \Delta/\Ee\rightarrow \Aa$ on $\Ee$ we get the following spectral sequence as a special case:
$$E_2^{p,q}\cong H^p(F\Bb, R^q (Fu\circ \nu)_*(T))\Rightarrow H^{p+q}_{Th}(\Ee, T)$$ 
which after identifying the various terms involved gives a spectral sequence for Thomason cohomology in terms of Baues-Wirsching cohomology:

\begin{theorem}\label{Thm+BW}
Let $\Ee$ and $\Bb$ be small categories and $u: \Ee\rightarrow \Bb$ a functor. Let $\Aa$ be a complete abelian category with exact products. Given a Thomason natural system $T: \Delta/\Ee\rightarrow \Aa$, there is a first quadrant cohomology spectral sequence 
$$E_2^{p,q}\cong H_{BW}^p(\Bb, R^q (Fu\circ \nu)_*(T))\Rightarrow H^{p+q}_{Th}(\Ee, T)$$
which is functorial with respect to natural transformations and where $R^q(Fu\circ \nu)_*=Ran^q_{Fu\circ \nu}$ is given as the $q$-th right satellite of $Ran_{Fu\circ \nu}$, the right Kan extension along $Fu\circ \nu$.
\end{theorem}

Again we can identify the terms in the $E_2$-page of the spectral sequence with concrete fiber data and get:

\begin{corollary}\label{Cor+BW}
Let $u: \Ee\rightarrow \Bb$ be a functor between small categories. Let $\Aa$ be a complete abelian category with exact products. Given a Thomason natural system $T: \Delta/\Ee\rightarrow \Aa$, there is a first quadrant cohomology spectral sequence 
$$E_2^{p,q}\cong H_{BW}^p(\Bb, \h^q(-/Fu\circ \nu, T\circ Q^{(-)}))\Rightarrow H_{Th}^{p+q}(\Ee, D)$$
which is functorial with respect to natural transformations and where  
$$\h^q(-/Fu\circ \nu, T\circ Q^{(-)})= \llim^q_{-/Fu\circ \nu} (T\circ Q^{(-)}): F\Bb\rightarrow \Aa.$$
\end{corollary}

In the special case that $T=\nu^*D$ for a Baues-Wirsching natural system $D: F\Ee\rightarrow \Aa$ on $\Ee$ we can further identify the $E_2$-term of the spectral sequence in Theorem \ref{Thm+BW} and get a spectral sequence for Baues-Wirsching cohomology:
$$E_2^{p,q}\cong H_{BW}^p(\Bb, R^q (Fu_*)(D))\Rightarrow H^{p+q}_{BW}(\Ee, D) $$

Similarly, we can identify the $E_2$-page in terms of local fiber data as in Corollary \ref{Cor+BW}.~This spectral sequence describes directly the functorial behaviour of Baues-Wirsching cohomology for functors between small categories.
(see also \cite[Theorem 1.14]{GNT}).\\

Dually, with similar arguments as above we can derive homological versions of the spectral sequences for Thomason homology of categories.

\begin{theorem}\label{H_*T}  Let $\Ee$ and $\Bb$ be small categories and $u: \Ee\rightarrow \Bb$ a functor. Let $\Aa$ be a cocomplete abelian category with exact coproducts. Given a contravariant Thomason natural system $T: \Delta/\Ee\rightarrow \Aa$ on $\Ee$, there is a third quadrant homology spectral sequence 
$$E^2_{p,q}\cong H^{Th}_p(\Bb, (L_q \Delta/u^*)(T))\Rightarrow H_{p+q}^{Th}(\Ee, T)$$
which is functorial with respect to natural transformations, where $L_q \Delta/u^*=Lan_q^{\Delta/u}$ is the $q$-th left satellite of $Lan^{\Delta/u}$, the left Kan extension along $\Delta/u$.
\end{theorem}

\begin{proof}
This is simply dual to the statement of Theorem \ref{H^*T} and follows from the dual Andr\'e homology spectral sequence involving the higher derived functors of $\clim^{\Delta/\Cc} D$ in the description of Thomason homology $H^{Th}_n(\Cc, T)$ (see also \cite{An} and \cite[Appendix II.3]{GZ}).
\end{proof}

We can identify the $E^2$-page of this spectral sequence and get:

\begin{corollary}\label{H_*T+GZ}  Let $\Ee$ and $\Bb$ be small categories and $u: \Ee\rightarrow \Bb$ a functor. Let $\Aa$ be a cocomplete abelian category with exact coproducts. Given a contravariant Thomason natural system $T: \Delta/\Ee\rightarrow \Aa$ on $\Ee$, there is a third quadrant cohomology spectral sequence 
$$E^2_{p,q}\cong H^{Th}_p(\Bb, \h_q(\Delta/u/-, T\circ Q_{(-)}))\Rightarrow H_{p+q}^{Th}(\Ee, T)$$
which is functorial with respect to natural transformations and where  
$$\h_q(\Delta/u/-, T\circ Q_{(-)})= \clim_q^{\Delta/u/-}(T\circ Q_{(-)}): \Delta/\Bb\rightarrow \Aa.$$
\end{corollary}

\begin{proof} We can identify the $E^2$-term in the above homology spectral sequence as follows (see \cite{An}, \cite{CE}, \cite{Hu})
$$Lan_q^{\Delta/u}(T)\cong \clim^{\Delta/u/\beta}_q T\circ Q_*$$ which gives the desired homology spectral sequence.
\end{proof}

In the special situation that $u$ is part of an adjoint pair or an equivalence of categories we will recover Proposition \ref{adjequivhom}.

Finally, in the particular case $\Cc=F\Bb$ with $u'=
{\nu}\circ \Delta/u$ for a contravariant Thomason natural system $T: \Delta/\Ee\rightarrow \Aa$ on $\Ee$ we get the homology spectral sequence 
$$E^2_{p,q}\cong H_p(F\Bb, L_q (
{\nu}\circ \Delta/u)_*(T))\Rightarrow H_{p+q}^{Th}(\Ee, T)$$ 
which after identifying the various terms gives the following spectral sequence for Baues-Wirsching homology, which is dual to the one of Theorem \ref{Thm+BW}:

\begin{theorem} 
Let $\Ee$ and $\Bb$ be small categories and $u: \Ee\rightarrow \Bb$ a functor. Let $\Aa$ be a cocomplete abelian category with exact coproducts. Given a contravariant Thomason natural system $T: \Delta/\Ee \rightarrow \Aa$, there is a third quadrant homology spectral sequence 
$$E^2_{p,q}\cong H^{BW}_p(\Bb, L_q (Fu\circ \nu_*)(T))\Rightarrow H_{p+q}^{Th}(\Ee, T)$$
which is functorial with respect to natural transformations and where $L_q(Fu\circ \nu)_*=Lan_q^{Fu\circ \nu}$ is given as the $q$-th left satellite of $Lan^{Fu\circ \nu}$, the left Kan extension along $Fu\circ \nu$.
\end{theorem}

We can again identify the terms in the $E^2$-page of the spectral sequence and obtain:

\begin{corollary}
Let $u: \Ee\rightarrow \Bb$ be a functor between small categories and $\Aa$ a cocomplete abelian category with exact coproducts. Given a contravariant Thomason natural system $T: \Delta/\Ee \rightarrow \Aa$, there is a third quadrant homology spectral sequence 
$$E^2_{p,q}\cong H^{BW}_p(\Bb, \h_q(Fu\circ \nu/-, T\circ Q_{(-)}))\Rightarrow H^{Th}_{p+q}(\Ee, T)$$
which is functorial with respect to natural transformations and where  
$$\h_q(Fu\circ \nu/-, T\circ Q_{(-)})= \clim_q^{Fu\circ \nu/-}(T\circ Q_{(-)}): \Bb^{op}\times \Bb\rightarrow \Aa.$$
\end{corollary}

As in the particular case $T=\nu^*D$ with a contravariant Baues-Wirsching natural system $D: F\Ee\rightarrow \Aa$ on $\Ee$ we can identify the $E^2$-term of this spectral sequence similarly as in Theorem \ref{Thm+BW} to obtain a spectral sequence for Baues-Wirsching homology of the form (see also \cite[Theorem 1.18]{GNT}):
$$E^2_{p,q}\cong H^{BW}_p(\Bb,  L_q (Fu_*)(D))\Rightarrow H_{p+q}^{BW}(\Ee, D)$$   

We can also derive cohomology and homology spectral sequences for all the other types of coefficient systems like $\Ee$-modules, local systems or trivial systems. The proofs go along the same lines of arguments as for Thomason natural systems, using the above ladder of categories and functors and the interpretation of the classical theories as special cases of Thomason cohomology and homology.

\subsection{Thomason (co)homology for Grothendieck fibrations}

We will now study the spectral sequences we have just constructed in the particular situation when the functor between small categories is a Grothendieck fibration.  Applying the spectral sequence constructions of the preceding paragraph to such a situation allows to identify the $E_2$-pages with simpler cohomology and homology data which keep track of the concrete fiber of the functor. These constructions also generalize similar spectral sequences for Baues-Wirsching cohomology as were discussed in \cite{GNT} and the line of arguments here is very similar.

For a given functor $u: \Ee\rightarrow \Bb$  between small categories $\Ee$ and $\Bb$ and a given object $b$ of $\Bb$, recall that the {\it fiber category}
$\Ee_b=u^{-1}(b)$ is the subcategory of $\Ee$ which fits into the following pullback diagram
$$
\xymatrix{\Ee_b\rto
\dto&\Ee\dto^u\\ {*}\rto^b&\Bb.}
$$
The objects of $\Ee_b$ are those objects of  $\Ee$ which map onto $b$ via the functor $u$ and the morphisms are given by those which map to the identity $1_b$.

We have the following notion of a fibration between small categories due to Grothendieck \cite[Expos\'e VI]{SGA}:
\begin{definition}
Let $\Ee$ and $\Bb$ be small categories. A {\it Grothendieck fibration} is a functor $u: \Ee\rightarrow \Bb$ such that the fibers
$\Ee_b=u^{-1}(b)$ depend contravariantly and pseudofunctorially on the objects $b$ of the category $\Bb$. The category $\Ee$ is also called a {\it category fibered over} $\Bb$.
\end{definition}

Let us recall here the following equivalent explicit description of a Grothendieck fibration \cite{SGA}, \cite[Vol.\ 2, 8.3.1]{Bo}. A functor 
$u: \Ee\rightarrow \Bb$ is a Grothendieck fibration if for each object $b$ of $\Bb$ the inclusion functor from the fiber into the comma category
\[
j_b: \Ee_b\to b/u, \quad e\mapsto (e, b\stackrel=\longrightarrow ue)
\]
is coreflexive, i.~e. has a right adjoint left inverse. Grothendieck fibrations are characterized as pseudofunctors. In fact, there is an equivalence of $2$-categories
$$\Fib(\Bb)\stackrel{\simeq}\leftrightarrow \PsdFun(\Bb^{op}, \Cat)$$
between the $2$-category of Grothendieck fibrations $\Fib(\Bb)$ over a small category $\Bb$ and the $2$-category of contravariant pseudofunctors $\PsdFun(\Bb^{op}, \Cat)$ from $\Bb$  to the category $\Cat$ of small categories.

From now on let us assume that the functor $u: \Ee\rightarrow \Bb$ is a Grothendieck fibration and $T: \Delta/\Ee \rightarrow \Aa$ a Thomason natural system on the category $\Ee$, where $\Aa$ is again a complete abelian category with exact products.

We get a local system 
$\h^q_{Th}(G(-), T|_{\Delta/\Ee_{(-)}}): \Bb\rightarrow \Aa$ from the associated pseudo\-functor $G: \Bb^{op} \rightarrow \Cat$ by assigning to every object $b$ of the category $\Bb$ the $q$-th Thomason cohomology of the category $G(b)$
$$\h^q_{Th}(G(-), T|_{\Delta/\Ee_{(-)}}): \Bb\rightarrow \Aa, \,\,\, b\mapsto H_{Th}^q(G(b), T|_{\Delta/\Ee_b})$$
where the coefficients are given by the Thomason natural system
$T|_{\Delta/\Ee_b}: \Delta/\Ee_b\rightarrow \Aa$.

For each object $b$ of the base category $\Bb$ 
we get now a cartesian diagram
$$
\xymatrix{\Ee_b\dto_{=}\rto^
{j_b}&b/u\dto^{Q^b}\\
\Ee_b\dto\rto^
{i_b}\dto&\Ee\dto^u\\ {*}\rto^b&\Bb.}
$$
Let $R_{b}$ denote the right adjoint functor of the inclusion functor $j_b$
and let $T_b$ denote the Thomason natural system on $b/u$ given by the composition
$$T_b: \Delta/ (b/u) \xrightarrow{\Delta/ R_b} \Delta/\Ee_b \xrightarrow{\Delta/ i_b} \Delta/\Ee \xrightarrow {T} \Aa.$$

The discussion in the last paragraph gives a first quadrant cohomology spectral sequence using the equivalence between Grothendieck fibrations and pseudofunctors:
$$E_2^{p,q}\cong H^p(\Bb, \h^q_{Th}(-/u, T_{(-)}))\Rightarrow H^{p+q}_{Th}(\Ee, T)$$

This spectral sequence is functorial with respect to $1$-morphisms, i.e.~natural transformations between Grothendieck fibrations. It generalizes a similar spectral sequence for Baues-Wirsching cohomology constructed before by Pirashvili-Redondo \cite{pr}.

To summarize, we have constructed the following particular cohomology spectral sequence:
\begin{theorem}\label{H^*T+G} 
Let $\Ee$ and $\Bb$ be small categories and let $u: \Ee\rightarrow \Bb$ be a Grothendieck fibration. Given a Thomason natural system $T: \Delta/\Ee\rightarrow \Aa$ on $\Ee$, where $\Aa$ is a complete abelian category with exact products, there is a first quadrant spectral sequence 
$$E_2^{p,q}\cong H^p(\Bb, \h^q_{Th}(-/u, T_{(-)}))\Rightarrow H^{p+q}_{Th}(\Ee, T)$$
which is functorial with respect to $1$-morphisms of Grothendieck fibrations. 
\end{theorem}

To identify the $E_2$-term of this spectral sequence with local data of the fiber category we introduce the following notion, corresponding to the property of {\it h-locality} for Baues-Wirsching natural systems introduced in \cite{pr}.

\begin{definition} Let $\Aa$ be a complete abelian category with exact products.
A Thomason natural system $T: \Delta/\Ee\rightarrow \Aa$ is called {\it local} if the adjoint functor $R_{b}$ of the inclusion functor $j_b: \Ee_b\to b/u$ induces an isomorphism in Thomason cohomology
$$H^q_{Th}(b/u, T_b)\cong H^q_{Th}(\Ee_b, T\circ\Delta/i_b)$$
for every $q$ and every object $b$ of the base category $\Bb$, i.~e. we have a natural isomorphism of local coefficient systems
$$\h^q_{Th}(-/u,  T_{(-)})\cong \h^q_{Th}(\Ee_{(-)}, T\circ \Delta/i_{(-)}).$$
\end{definition}

Identifying the $E_2$-page of the above spectral sequence, we get the following spectral sequence:

\begin{theorem}\label{H^*T+Glocal} Let $\Aa$ be a complete abelian category with exact products.
Let $\Ee$ and $\Bb$ be small categories and $u: \Ee\rightarrow \Bb$ be a Grothendieck fibration. Given a local Thomason natural system $T: \Delta/\Ee\rightarrow \Aa$ on $\Ee$, where $\Aa$ is a complete abelian category with exact products, there is a first quadrant spectral sequence 
$$E_2^{p,q}\cong H^p(\Bb, \h^q_{Th}(\Ee_{(-)}, T\circ \Delta/i_{(-)}))\Rightarrow H^{p+q}_{Th}(\Ee, T)$$
with the local coefficient system 
$$\h^q_{Th}(\Ee_{(-)}, T\circ \Delta/i_{(-)}): \Bb\rightarrow \Aa, \,\, b\mapsto H^q_{Th}(\Ee_b, T\circ \Delta/i_b).$$
Furthermore, the spectral sequence is functorial with respect to $1$-morphisms of Grothendieck fibrations.
\end{theorem}

\begin{proof} 
The construction of the cohomology spectral sequence in Theorem \ref{H^*T+G} and the definition of a local Thomason natural system $T: \Delta/\Ee\rightarrow \Aa$ on the total category $\Ee$ give the desired identification of the $E_2$-page of the spectral sequence.
\end{proof}

Using the obvious dual notions of all the above constructions we can derive a homology version of the spectral sequence for a Grothendieck fibration:

\begin{theorem}\label{H_*T+Glocal} Let $\Aa$ be a cocomplete abelian category with exact coproducts.
Let $\Ee$ and $\Bb$ be small categories and $u: \Ee\rightarrow \Bb$ be a Grothendieck fibration. Given a colocal contravariant Thomason natural system $T: \Delta/\Ee\rightarrow \Aa$ on $\Ee$, where $\Aa$ is a cocomplete abelian category with exact coproducts, there is a third quadrant spectral sequence 
$$E^2_{p,q}\cong H_p(\Bb, \h_q^{Th}(\Ee_{(-)}, T\circ \Delta/i_{(-)}))\Rightarrow H_{p+q}^{Th}(\Ee, T)$$
with the local coefficient system 
$$\h_q^{Th}(\Ee_{(-)}, T\circ \Delta/i_{(-)}): \Bb\rightarrow \Aa, \,\, b\mapsto \h_q^{Th}(\Ee_b, T\circ \Delta/i_b).$$
Furthermore, the spectral sequence is functorial with respect to $1$-morphisms of Grothendieck fibrations.
\end{theorem}

Using the ladder of categories and functors we can also derive various cohomology and homology spectral sequences for Baues-Wirsching natural systems, bimodules, modules, local systems or trivial systems.
These spectral sequences will then relate the associated classical cohomology and homology theories in a Grothendieck fibration. 

\vspace*{0.3cm} {\it Acknowledgements.} 
The first author was partially supported by the grants 
MTM2010-15831, 
MTM2010-20692, 
and SGR-1092-2009, 
and the third author by 
MTM2010-15831 
and SGR-119-2009. 
The second author would like to thank the Barcelona Algebraic Topology Group at the Universitat Aut{\`o}noma de Barcelona (UAB) for the kind invitation and financial support
under the grant 
SGR-1092-2009. 

\bigskip

\end{document}